\documentclass[11pt, a4paper]{amsart}
\usepackage{a4wide}
\usepackage{mathrsfs, amssymb,amsmath,url,amsthm}

\title[A closed algebra with a non-Borel clone]{A closed algebra with a non-Borel clone and\\ an ideal with a Borel clone}
\date{\today}

\author[M.~Goldstern]{Martin Goldstern}
\address{Algebra\\TU Wien\\Wiedner Hauptstra\ss e 8-10/104\\A-1040 Wien, Austria}
\email{goldstern@tuwien.ac.at}\urladdr{http://www.tuwien.ac.at/goldstern/}
 \thanks{Research of the first author supported by FWF project P 21968-N13.}
 
\author[M.~Pinsker]{Michael Pinsker}
    \address{\'{E}quipe de Logique Mathématique\\ Universit\'{e} Denis Diderot -- Paris 7\\
	UFR de Math\'{e}matiques\\
	75205 Paris Cedex 13, France}
    \email{marula@gmx.at}
    \urladdr{http://dmg.tuwien.ac.at/pinsker/}
    \thanks{Research of the second author supported by an APART-fellowship of the Austrian Academy of Sciences.}
\author[S.~Shelah]{Saharon Shelah}
    \address{Institute of Mathematics\\ The Hebrew University of Jerusalem\\91904 Jerusalem, Israel, and Department of Mathematics, Rutgers University, New Brunswick, New Jersey 08854}
    \email{shelah@math.huji.ac.il}
    \urladdr{http://shelah.logic.at}
\thanks{Research of the third author partially supported by NSF grant no: DMS 1101597, and
GIF (German-Israeli Foundation for Scientific Research \& Development) Grant no. 963-98.6/2007. Publication 994 on Shelah's list.}

\theoremstyle{plain}

    \newtheorem{thm}{Theorem}

    \newtheorem{lem}[thm]{Lemma}

    \newtheorem{fact}[thm]{Fact}

\theoremstyle{definition}

    \newtheorem{defn}[thm]{Definition}

\newcommand{\on}{{\upharpoonright}}
\newcommand{\nin}{\notin}
\DeclareMathOperator{\id}{id}

\newcommand{\inv}{^{-1}}

\newcommand{\un}{^{(n)}}
\newcommand{\uk}{^{(k)}}

\newcommand{\ignore}[1]{}

\newcommand{\cl}[1]{\langle #1 \rangle}

\renewcommand{\O}{{\mathscr O}}
\newcommand{\On}{{\mathscr O}^{(n)}}
\newcommand{\Ok}{{\mathscr O}^{(k)}}
\newcommand{\Oo}{{\mathscr O}^{(1)}}
\newcommand{\Ot}{{\mathscr O}^{(2)}}

\DeclareMathOperator{\pol}{Pol}

\newcommand{\C}{{\mathscr C}}

\newcommand{\F}{{\mathscr F}}

\newcommand{\To}{\rightarrow}

\newcommand{\A}{{\mathfrak A}}
\newcommand{\B}{{\mathscr B}}

\newcommand{\G}{{\mathscr G}}

\newcommand{\pfeil}[1]{\ \mathop{\longrightarrow}\limits^{#1}\ } 
\newcommand{\NN}{{\mathbb N}}
\newcommand{\CC}{\C_{I_{\bar d=0}}}
\newcommand{\OO}{{\mathscr O}}
\theoremstyle{remark}

\begin{document}

\begin{abstract}
	Algebras on the natural numbers and their clones of term operations can be classified according to their descriptive complexity. We give an example of a closed algebra which has only unary operations and whose clone of term operations is not Borel. Moreover, we provide an example of a coatom in the clone lattice whose  obvious definition via an ideal of subsets of natural numbers would suggest that it is complete coanalytic, but which turns out to be a rather simple Borel set. Our results solve Problems~E and~N from~\cite{GoldsternPinsker} and Problem~40 from~\cite{BGHP}.

\end{abstract}

\maketitle

\section{Two problems about clones on $\NN$}

\subsection{Descriptive set theory of algebras and clones on $\NN$}\label{subsect:1}

Let $X$ be a set, and denote for all $n\geq 1$ the set $X^{X^n}$ of all functions on $X$ in $n$ variables  by $\On$. Then $\O:=\bigcup_{n\geq 1} \On$ is the set of all finitary functions on $X$. A \emph{clone} is a subset $\C$ of $\O$ which contains all projections (i.e., all functions satisfying an equation of the form $f(x_1,\ldots,x_n)=x_k$) and which is closed under composition, i.e., for all $n,m\geq 1$, all $n$-ary $f\in \C$, and all $m$-ary $g_1,\ldots,g_n\in\C$, the $m$-ary function $f(g_1(x_1,\ldots,x_m),\ldots,g_n(x_1,\ldots,x_m))$ is also an element of $\C$. In other words, $\C$ is required to be closed under building of \emph{terms} from its functions. The latter perspective shows that clones arise naturally as sets of term functions of algebras with domain $X$; in fact, the clones on $X$ are \emph{precisely} the sets of term functions of such algebras. Since many properties of an algebra (e.g., subalgebras, congruences) depend only on the clone of the algebra, clones are in that sense canonical representatives of algebras, and have been studied intensively in the literature; for a monograph on clones, see~\cite{Szendrei}.

While clones arise in this way on base sets $X$ of arbitrary (finite or infinite) cardinality, there is an additional perspective on clones from the viewpoint of descriptive set theory that can only be enjoyed on a countably infinite base set, as we will outline in the following. For notational and conceptual convenience, let us identify $X$ with the set of natural numbers $\NN$. 
Then, for every fixed $n\geq 1$, we can view $\On=\NN^{\NN^n}$ as a topological space whose topology is naturally given by equipping $\NN$ with the discrete topology and viewing $\NN^{\NN^n}$ as a product space. This space is homeomorphic to the \emph{Baire space} (the metric space on $\NN^\NN$ in which two functions are closer the later they start to differ), and a function $f\in\On$ is in the closure of a set $\F\subseteq \On$ iff for every finite subset of $\NN^n$ there exists $g\in\F$ which agrees with $f$ on this set. The set $\O$ then becomes the sum space of the spaces $\On$, i.e., the open subsets of $\O$ are those sets $\F$ for which the $n$-ary fragment $\F\un:=\F\cap \On$ is  open in $\On$ for every $n\geq 1$. This space is itself homeomorphic to the Baire space, and is in particular a \emph{Polish space}, i.e., a separable topological space whose topology is generated by a complete metric (confer the textbook~\cite{Kechris}).  

The latter fact allows for the use of the notions of descriptive set theory on $\O$. A subset of a Polish space is called \emph{analytic}
 iff it is the continuous image of a closed subset of the Baire space.  For example, all Borel sets are analytic. The \emph{coanalytic}
 subsets of a Polish space are defined to be the complements of analytic sets. A coanalytic set $Y$ in a Polish space $S$ is called \emph{complete
coanalytic} iff for every Polish space $S'$ and every coanalytic
set $Y'$ therein, $Y'$ is the preimage of $Y$ under some
continuous function from $S'$ to $S$. In every uncountable Polish space, in particular in $\O$, there are analytic sets which are not coanalytic. A complete coanalytic set can
therefore not be analytic.

A central theme of descriptive set theory is the  investigation of
the complexity of subsets of Polish spaces. In this descriptive complexity hierarchy,  the simplest sets are the
closed and the open sets. Borel sets are still 
considered relatively simple, and in fact, most sets of real numbers that appear in
analysis are Borel sets.  Analytic sets are more complicated
than Borel sets (similar to the difference between recursively
enumerable and recursive sets), and coanalytic sets are considered
to be slightly more complicated.

As subsets of the Polish space $\O$, sets of finitary functions on $\NN$ can thus be classified according to this descriptive complexity. In particular, this applies to the functions of an algebra and to clones; we then call algebras and clones open, closed, Borel, analytic, etc. The clones that arise most naturally are at the lowest level of the descriptive hierarchy: the \emph{closed clones} are precisely the \emph{polymorphism clones} of relational structures with domain $\NN$, i.e., the sets of finitary functions preserving all relations of some relational structure. Similarly to automorphism groups, polymorphism clones contain information about their corresponding relational structure, and are investigated in order to derive properties of this structure; we refer to~\cite{BodChenPinsker,BP-reductsRamsey, BodPin-Schaefer-STOC} for applications of closed clones in model theory and theoretical computer science. 

Much further up in complexity, the first author of the present paper proved that a certain important clone containing $\Oo$, called $\pol(T_2)$, is complete coanalytic in order to show that this clone could not be the term clone of any algebra which has only countably many non-unary functions~\cite{Goldstern-analytic}: for any algebra which has all functions in $\Oo$ and countably many non-unary functions is Borel, and the clone of term operations of a Borel algebra is always analytic. So far, no algebraic proof of the above result is known, and thus descriptive set theory is not only a way of classifying clones on $\NN$, but also a tool for proving theorems about such clones.

\begin{fact} If an algebra is Borel or analytic, then its clone of term operations 
is analytic.
\end{fact}

The mentioned exploitation of this bound on the increase in descriptive complexity when passing from an algebra to its clone of term operations inspired the authors of the survey paper~\cite{GoldsternPinsker} to ask whether there exists a Borel algebra whose term clone is not Borel; the problem was stated as Problem~N in the survey. Here, we will show that there exists an algebra which

\begin{itemize}
	\item has contains only unary functions,
	\item is closed, and
	\item has a term clone which is not Borel,
\end{itemize}
providing an affirmative answer to this problem.

\subsection{A false coanalytic ideal clone on $\NN$}\label{subsect:2}

A large class of clones on $\NN$ are \emph{ideal clones} (the class provides, in particular, $2^{2^{\aleph_0}}$ coatoms in the lattice of clones without use of the Axiom of Choice~\cite{BGHP}). Let $I$ be an ideal of subsets of $\NN$, that is, a downward closed set of subsets which is closed under finite joins. Then the set $\C_I$ of all finitary functions on $\NN$ which send powers of sets in $I$ to sets in $I$ is a clone; from their definition which universally quantifies over all subsets of the natural numbers, clones of this form can be expected to be rather up in the descriptive hierarchy. One natural ideal on $\NN$ is the following. For each set $A \subseteq \NN$, the upper density 
$\bar d(A)$ is defined as $$  \bar d(A):= \limsup\limits_{n\to
\infty} \frac{|A\cap
[0,n) |}{n}. $$ 
The family of sets of upper density 0 forms an ideal $I_{\bar d=0}$. The corresponding clone $\C_{I_{\bar d=0}}$ was studied by the authors of~\cite{BGHP}, who at the time could not determine whether or not the clone was \emph{precomplete}, i.e.,  a coatom in the lattice of all clones on $\NN$. Moreover, the second author of that paper conjectured that $\C_{I_{\bar d=0}}$ is, just like the clone $\pol(T_2)$ mentioned above, complete coanalytic, in accordance with its obvious definition quantifying over subsets of $\NN$; this would have explained the difficulties when trying to decide whether or not the clone is a coatom. We disprove this conjecture by showing that 
$\C_{I_{\bar d=0}}$ is in fact a Borel set of low complexity. Moreover, we show that $\C_{I_{\bar d=0}}$ is indeed precomplete, solving Problem~40 of~\cite{BGHP} (also known as Problem~E in
 in the survey paper~\cite{GoldsternPinsker}).
 
 \subsection{Summary and organization of the paper}

 We thus provide in this paper an example of a rather complex (non-Borel) clone which comes from an algebra that is rather simple (closed, and moreover unary) (Section~\ref{sect:2}), and an example of a rather simple (low Borel) clone which seems to be complex in the sense that its obvious definition by means of a non-trivial ideal does not suggest it is Borel, and in the sense that determining its precompleteness is a relatively hard task (Section~\ref{sect:3}). Our results solve Problems~N and~E from~\cite{GoldsternPinsker}; the latter problem has been stated as Problem~40 in~\cite{BGHP}.

\section{A closed algebra with a non-Borel clone}\label{sect:2}

\begin{thm}
	There exists a closed algebra on $\NN$ whose clone of term operations is not Borel. Moreover, this algebra can be chosen to contain only unary functions.
\end{thm}

\begin{proof}
Write $\NN$ as a disjoint union $\{0\}\cup T_x \cup T_y\cup A_x\cup A_y$, where $T_x$, $T_y$, $A_x$ and $A_y$ are infinite. Consider the space $T_x^{A_x}$ of all functions from $A_x$ to $T_x$, equipped with the metric that makes it the Baire space, and consider likewise $T_y^{A_y}$. Then there exists a closed subset $B$ of the product space $T_x^{A_x}\times T_y^{A_y}$ whose projection onto the first coordinate is analytic but not Borel; see for example the textbook~\cite{Kechris}.

Let $\F$ contain the identity function $\id$ on $\NN$ plus the set of all functions $f$ in $\NN^\NN$ such that:
\begin{itemize}
	\item $f$ is the identity on $\{0\}\cup T_x\cup T_y$, and
	\item the pair $(f\on_{A_x}, f\on_{A_y})$ is an element of $B$.
\end{itemize}

Then $\F$ is a transformation monoid since $f(f'(x))=f'(x)$ for all $f, f'\in \F$ such that $f'$ is not the identity function. Moreover, it is clearly a closed subset of $\Oo$ since the set $B$ is closed.

Now let $h:\NN\To\NN$ defined by
$$
	h(n):=\begin{cases} n,& n\in \{0\}\cup T_x\cup T_y\cup A_x\\
	0,& n\in A_y.
	\end{cases}
$$

Set $\G$ to contain all unary functions that can be composed from elements of the set $\{h\}\cup\F$. Then $\G$ is the disjoint union of $\{h\}$, $\F$, and the set $\G'$ of all functions $g$ in $\Oo$ such that
\begin{itemize}
	\item $g$ is the identity on $\{0\}\cup T_x\cup T_y$, and
	\item $g\on {A_x}$ is in the projection of $B$ onto the first coordinate, and
	\item $g(n)=0$ for all $n\in A_y$.
\end{itemize}

To see this, observe that $h\circ f=f$ for all $f\in \F$, and $f\circ h$ is the element $g$ of $\G'$ which agrees with $f$ on $A_x$.

Since the projection of $B$ onto the first coordinate is not Borel, $\G'$ is not Borel. Hence $\G$, as  the disjoint union of $\G'$ with the closed set $\{h\}\cup\F$, is not Borel either. Therefore, taking $\{h\}\cup\F$ as the functions of the algebra proves the theorem.
\end{proof}

\section{The clone preserving zero upper density}\label{sect:3}

We now investigate the clone $\C_{I_{\bar d=0}}$ of all functions which preserve the ideal of sets of upper density $0$. We first show that it is Borel, and then that it is precomplete.

\subsection{The complexity of $\C_{I_{\bar d=0}}$}

In this part we give the prove of the following theorem.

\begin{thm}\label{thm:Borel}
	$\C_{I_{\bar d=0}}$  is Borel.
\end{thm}

\begin{defn}
Let $k\geq 1$ and let $f\in\Ok$.  Each permutation $\pi$ of $\{1,\ldots, k\}$ induces 
a function $f_\pi\in\Ok$ by setting $f_\pi(x_1,\ldots, x_k) := f(x_{\pi(1)}, \ldots, x_{\pi(k)})$. Moreover, for each $0\leq \ell  < k$, each tuple $\bar a = (a_1,\ldots, a_\ell)\in\NN^\ell$ induces a $(k-\ell)$-ary
function $f_  {\bar a }$ by setting $f_  {\bar a }(y_1,\ldots, y_{k-\ell}) := 
f(a_1,\ldots, a_\ell, y_1,\ldots, y_{k-\ell})$. We call each 
function $f_{\pi,\bar a}$ a \emph{shadow} of $f$. When $\ell>0$, then we call $f_{\pi,\bar a}$ a \emph{proper shadow} of $f$. The functions $f_\pi$, which are just the functions $f_{\pi,\bar a}$ for a tuple $\bar a$ of length $\ell=0$, and in particular $f$ itself, are called \emph{improper shadows} of $f$. 
\end{defn}

Observe that proper shadows of $f$ have strictly smaller arity than $f$. If $f$ is unary, then it has no proper shadows.

\begin{defn}
Let $k\geq 1$. We call $f\in\Ok$ \emph{minimal} if the following hold: 
\begin{itemize}
\item $f\notin \CC$;
\item every proper shadow of $f$ is in $\CC$. 
\end{itemize}
\end{defn}

We will show that $\CC$ is a Borel set as follows: a function $f$ is not contained in $\CC$ if and only if it has a minimal (proper or improper) shadow. Having a minimal shadow will turn out to be equivalent to having what is called a \emph{bad} function as a  shadow. The set of bad functions has a nice definition which makes it a Borel set, and having a bad shadow is in turn easily shown to be a Borel property, proving the theorem.

\begin{defn}  
Let $k\geq 1$ and $f\in\Ok$. We say that $f$ is \emph{bad}
iff the following holds: There exists a rational number $\varepsilon >0$ such that for  all $i\in\NN$ 
there  are $n, t\geq i$ and $A\subseteq [i,n)$ with the following properties:
   \begin{itemize}
   \item $A$ is sparse with respect to $i$:\;  $|A\cap  [0,r)| \le  \frac{1}{2^{i}}\cdot r $ for all $r\in \NN$;
   \item $f[A^k]$ is dense in $[0,t)$ with respect to $\varepsilon$:\; $|f[A^k]\cap [0,t)| \ge \varepsilon \cdot t$. 
   \end{itemize}
\end{defn}
 
\begin{lem}\label{bad.not.C}
Let $k\geq 1$, and let $f\in\Ok$ be bad.  Then $f\notin \CC$.  
\end{lem}

\begin{proof}
Let $\varepsilon$ as in the definition of badness, and set $n_0:=t_0:=0$. By inductively applying the definition of badness to $i>\max(n_{j-1},t_{j-1})$, we can find tuples $(n_j, t_j, A_j)$ for all $j\geq 1$ 
such that
\begin{itemize} 
\item $0< n_1<  n_2 < \cdots$;
\item $0< t_1 < t_2 < \cdots $;
\item $A_j \subseteq [n_{j-1},n_j)$ and $|A_j\cap [0, r)| \le  \frac{r}{2^{{n_{j-1}}}}$ for all $r\in\NN$; 
\item $ |f[A_j^k]\cap [0, t_j)| \ge  \varepsilon \cdot t_j$. 
\end{itemize} 
Now let $A:= \bigcup_{j\geq 1} A_j$. To see that $A$ has upper density~$0$, let any rational number $\delta>0$ be given; we will find $s\in\NN$ such that $\frac{1}{m}|A\cap [0, m)|<\delta$ for all $m>s$. To this end, pick a natural number $v>0$ such that $\frac{1}{2^{v}}<\frac{\delta}{2}$. Now pick $s\in\NN$ such that $\frac{n_{v-1}}{s}<\frac{\delta}{2}$. Then, for $m>s$, we have
\begin{align*}
\frac{1}{m}|A\cap [0, m)| 
&\leq \frac{1}{m} \left ( n_{v-1}+ |A\cap [n_{v-1},m)| \right )
<  \frac{\delta}{2}+ \frac{1}{m} \sum_{v\leq j} |A_j\cap [n_{j-1},m)| \leq\\
&\leq \frac{\delta}{2}+ \frac{1}{m} \sum_{v\leq j} \frac{m}{2^{n_{j-1}}} \leq \frac{\delta}{2}+ \frac{1}{m} \sum_{v\leq j} \frac{m}{2^{j-1}}
\leq \frac{\delta}{2}+ \frac{1}{2^v}<\delta.
\end{align*}
On the other hand, $f[A^k]$ has upper density of at least $\varepsilon$, since $|f[A^k]\cap[0,t_j)|\geq \varepsilon \cdot t_j$ for all $j>0$. Hence, $f\notin \CC$.
\end{proof}

We will use the following lemma in order to show that minimal functions are bad.

\begin{lem}\label{ak}
Let $k\geq 1$, and let $f\in\Ok$ be minimal. Let $B\subseteq \NN$ be so that the fact ``$f\notin \CC$''
is witnessed by $B$, i.e., $B$ has upper density 0, but 
 $f[B^k]$ has positive upper density. Then for each $i\geq 0$ the set $B \setminus [0,i]$ also witnesses 
that $f\notin \CC$; in fact, the sets $f[B^k]$ and
$f[(B\setminus [0,i])^k]$ have equal upper density.
\end{lem}

\begin{proof}
If $k=1$, then $f[(B\setminus [0,i])^k]=f[B]\setminus f[[0,i]]$, so the statement follows immediately since $f[[0,i]]$ is finite. Now assume $k\geq 2$.
Let $S$ be the set of all proper shadows $f_{\pi,\bar a}$
of $f$ for which the tuple $\bar a$ contains only elements of $[0,i]$; since $k\geq 2$, this set is non-empty. By the minimality of $f$, each $f'\in S$ is an element of $\CC$, and so the set $f'[B^{k-{\ell_{f'}}}]$ (for appropriate $1\leq \ell_{f'}<k$) has upper density~$0$. Since $S$ is finite, the set 
$$D:= f[B^k] \setminus  \bigcup_{f'\in S}  f'[B^{k-{\ell_{f'}}}]$$
still has the same positive upper density as $f[B^k] $.   It remains to check that 
$f[(B\setminus [0,i])^k] \supseteq D$.  Let $d\in D$.  
Then $d$ cannot be written as $d=f(b_1,\ldots, b_k)$ with all  $b_i\in B$
and at least one $b_i$ in $[0,i]$, as any such $d$ would be in
$f'[B^k]$ for some proper shadow $f'\in S$. On the other hand, $d$ can be 
written as $d=f(b_1,\ldots, b_k)$ with all  $b_i\in B$.   Hence 
$d\in f[(B\setminus [0,i])^k] $.
\end{proof}

\begin{lem}\label{minimal.bad}
Let $f\in \O$ be minimal.  Then $f$ is bad. 
\end{lem}

\begin{proof}   
Write $k$ for the arity of $f$. Let $B\subseteq \NN$ be so that $\bar d(B)=0$ and $\bar d(f[B^k])>0$, and let $\varepsilon$ be a positive rational number such that $\bar d(f[B^k])> \varepsilon$. Given $i\in\NN$  we have to find $n,t,A$ as in the definition of badness. 

Since $\bar d(B)=0$, we can pick $m\geq i$ so large that $|B\cap [0,j)|\le \frac{1}{2^{i}}\cdot j$
for all $j\ge m$. Set $D:=B\cap [m,\infty)$. Then by Lemma~\ref{ak}, $\bar d(f[D^k])= \bar d(f[B^k]) >\varepsilon$. Hence, we can find $t\geq i$ such that $f[D^k]\cap
[0,t)$ has size at least $ \varepsilon\cdot t$. Now choose $n\geq m$ such that
$f[D^k]\cap [0,t) = f[(D \cap [0,n))^k]\cap [0,t)$. Finally, set $A:= D \cap [0,n) =  B \cap [m,n)$.
\end{proof} 

\begin{lem}\label{lem:equi}
Let $f\in\O$.  The following are equivalent: 
\begin{enumerate}
\item[(a)]  $f\notin \CC$. 
\item[(b)] There exists a shadow of $f$ which is minimal.  
\item[(c)] There exists a shadow of $f$ which is bad. 
\item[(d)] There exists a shadow of $f$ which is not in $\CC$. 
\end{enumerate}
\end{lem}
\begin{proof}[Proof of (a)$\Rightarrow$(b)]
 Let $S$ be the set of shadows of $f$ which are
not in $\CC$.   Let $g\in S$ have minimal arity.  Then $g$ is minimal.
\end{proof}

\begin{proof}[Proof of (b)$\Rightarrow$(c)]
  Every minimal function is bad, by Lemma~\ref{minimal.bad}.
\end{proof}

\begin{proof}[Proof of (c)$\Rightarrow$(d)]
A bad function cannot be in $\CC$, by Lemma~\ref{bad.not.C}.
\end{proof}

\begin{proof}[Proof of (d)$\Rightarrow$(a)]
  Let $f_{\pi,\bar a}$ be a shadow of $f$ which is not in $\CC$, and let $B\subseteq\NN$ be a set of upper density~$0$ which is sent to a set of positive upper density under $f_{\pi,\bar a}$. Let $B'$ be the set obtained by adding all entries of the tuple $\bar a$ to $B$. Then $B'$ still has upper density~$0$, and $f$ sends $B'$ to a set of positive upper density.
\end{proof}

\begin{proof}[Proof of Theorem~\ref{thm:Borel}]
Clearly, the set $\B$ of bad functions is Borel since its definition only quantifies over natural and rational numbers. We show that $\CC\uk$ is Borel for all $k\geq 1$. For each $\pi$ in the set $S(\{1,\ldots,k\})$ of all permutations of $\{1,\ldots,k\}$ and each tuple $\bar a = (a_1,\ldots, a_\ell)\in\NN^\ell$, where $0\leq \ell  < k$, the mapping $\kappa_{\pi,\bar {a}}$ from $\Ok$ to $\O$ which sends every $f\in \Ok$ to $f_{\pi,\bar {a}}$ is continuous. By Lemma~\ref{lem:equi},
$$
	\Ok\setminus\CC=\bigcup_{\pi\in S(\{1,\ldots,k\})}\ \bigcup_{0\leq \ell  < k}\ \bigcup_{\bar a\in \NN^\ell} \kappa_{\pi,\bar a}\inv[\B]
$$
Since $\B$ is Borel, each of its continuous preimages $\kappa_{\pi,\bar a}\inv[\B]$ is Borel, and so is the countable union of these sets.

\end{proof}

\subsection{Precompleteness of $\CC$}

We will now show the following.

\begin{thm}\label{thm:precomplete}
 $\C_{I_{\bar d=0}}$ is \emph{precomplete}, i.e., a coatom of the lattice of all clones on $\NN$.
\end{thm}

The strategy is the following: we will first show that  every set $B$ of  positive upper density can be mapped by a unary function from $\CC$ onto a ``large set'', that is, a 
set containing infinitely many intervals of the form $[n,2n]$.   We then show that
for every large set  $C$ there is a set $D$ of upper density~$0$ such that
$C\times D$ can be mapped 
onto all of $\NN$ by a binary function from $\CC$. These two facts together imply that every function $g\nin\CC$, together with $\C$, generates (by building terms) a function $g'$
mapping set of upper density~$0$ onto $\NN$; and it is well-known that the only clone containing
 $\CC\cup \{g'\}$ is $\O$. Hence, the only clone containing $\CC\cup \{g\}$ is $\O$ as well, and thus $\CC$ is a coatom since $g$ was an arbitrary function outside $\CC$.

\begin{lem} \label{easy.2}
 $T\subseteq \NN$ has upper density 0 if the limit of $\frac{ |  T \cap [ 2^k, 2^{k+1})|}{2^k}$, where $k$ goes to infinity, 
exists and equals 0.
\end{lem}

\begin{proof} 
Easy.
\end{proof}

\begin{lem}\label{easy.3}
Let $\varepsilon>0$ be a rational number, and  $f:\NN\to \NN$ satisfy $f(n) \ge
n\cdot\varepsilon$ for all $n\ge 0$.  Then for all $A \subseteq \NN$
we have $\bar d(f[A]) \le \frac{1}{\varepsilon}\cdot\bar d (A)$.  In particular, 
$f\in \CC$. 
\end{lem}

\begin{proof}  For each $n\geq 1$ we have 
$$
\frac1n \bigl|f[A] \cap [0,n) \bigr| 
\le \frac 1n \bigl|f\bigl[A\cap [0,\frac{n}{\varepsilon})\bigr]\bigr| 
\le  \frac 1{n} \bigl | A \cap [0, \frac{n}{\varepsilon})\bigr |
= \varepsilon\cdot \frac 1{\frac n\varepsilon} \bigl | A \cap [0, \frac{n}{\varepsilon})\bigr |
$$ 
\end{proof}

The next lemma has been shown, for example, in~\cite{BGHP} and~\cite{CH01}, but we include the proof for the reader's convenience. In the following, for $\F\subseteq \O$ we write $\cl{\F}$ for the clone of all term operations over $\F$, i.e., the smallest clone containing $\F$.
 
\begin{lem} \label{known.1} 
Let $g\notin \CC$.  Then $\langle \CC\cup \{g\}\rangle$ contains a unary function which is not in $\CC$.
\end{lem} 
\begin{proof} 
Let $g$ be $k$-ary, where $k\geq 1$. Since  $g\notin \CC$, there is an infinite set $A\subseteq \NN$ of upper density $0$ such that 
 $g[A^k]$ has positive upper density. 
Let $f_1, \ldots, f_k$ be functions from $A$ into $A$ such 
that the function $n\mapsto (f_1(n), \ldots, f_k(n))$ is
a bijection from $A$ onto $A^k$.   If we set $f_i(n)=0$ 
for all $n\notin A$, then $f_i\in \CC$ for 
all $1\leq i\leq k$ since the range of $f_i$ is contained in a set of upper density~$0$.  The unary function $h(n):= g(f_1(n), \ldots, f_k(n))$
now maps $A$ onto  $g[A^k]$. 
\end{proof} 

 Slightly modifying the usual Landau symbol, we 
 will in the following write $O(x)$ for any quantity  in the interval  $[0, x]$, for any rational number $x\geq 0$ . 

\begin{lem}\label{n2n}
Let $B\subseteq \NN$ have positive upper density. 
Then there is a unary function $f\in \CC$ and a strictly  
increasing sequence $(n_i)_{i\geq 1}$ of natural numbers such that $f[B]\supseteq \bigcup_{i\geq 1} [n_i, 2n_i) $.
\end{lem}

\begin{proof}
Fix a positive natural number $e$ such that $\bar d(B)>\frac{3}{e}$.  We first claim that 
there are infinitely many $n\in\NN$ with $|B\cap [n, e  n)| \ge n $. 
So let $m\in\NN$ be given; we will find $n\ge  \lfloor\frac{m}{e} \rfloor$ with this property. 

Since $B$ has positive upper density, we can increase $m$ such that $|B\cap [0,m)| \ge \frac{3}{e} m$, and 
also such that $m > 2e^2$.  Now 
let $n:= \lfloor \frac me \rfloor$; then $ m = en+ O(e)$, and $e < n$.
So we have 
$$ | B\cap [n, en) | \ge | B\cap [ 0,  m) | -  |[0,n)| -  | [en,m)| \ge 
\frac{3}{e} m - n - O( e)  \geq \frac{3}{e} en -2n  = n .$$

Now we choose an infinite strictly increasing sequence $(n_i)_{i\geq 1}$ of natural numbers such that the intervals $I_i:= [n_i, en_i]$
are disjoint, and all $n_i$ have the above property.   We  can then find a
function $f\in \Oo$ with the following properties: $f(x) \ge \frac{x}{e}$ for all $x\in\NN$ (hence
$f\in \CC$), and  $f[B\cap I_i] \supseteq [n_i, 2n_i)$ for all $i\geq 1$. 
\end{proof}

The next lemma is the crucial step of the proof of Theorem~\ref{thm:precomplete}.

\begin{lem} \label{ab.lemma}
Let $(n_i)_{i\geq 1}$ be a strictly increasing sequence of natural numbers, and set 
$C:= \bigcup_{i\ge 1} [n_i, 2n_i)$.   Then there
exists a set $D\subseteq \NN$ of upper density~$0$ and  a binary function $h\in \CC$
such that $h[ C\times D ] = \mathbb N $. 
\end{lem}

\begin{proof}   
For notational simplicity we will aim for a function $h$ such that
$h[ C\times D ] \supseteq\NN\setminus\{1\}$.  It is then easy to modify $h$ to obtain $h[ C\times D ] = \mathbb N $.

\begin{itemize}
\item Set  $n_0:= 2$; then for each natural number $k \ge 1$, the natural number $i(k) := \max \{ i\geq 0:  n_i \le 2^k \}$ is well-defined.  Clearly the sequence
 $(i(k))_{k\geq 1}$ is weakly increasing and diverges to infinity. 
\item For all $k\geq 1$, let $d_k:= \lceil \frac {2^k }{n_{i(k)}} \rceil  $.  Then 
 $d_k = \frac {2^k }{n_{i(k)}} + O(1) $, and $ 1 \le d_k \le 2^{k-1}$.
 \item For all $k\geq 1$, let $D_k$ be the interval $(2^k-d_k , 2^k]$.   Note that these
 intervals are disjoint,
   and that $|D_k|=d_k $. 
\item Let $D := \bigcup_{k\geq 1}  D_k $. 
\end{itemize}

To check that  $D$ has density~$0$ we use Lemma~\ref{easy.2}:  clearly 
$\frac{1}{2^{k}}|D\cap [ 2^{k}, 2^{k+1})|  = \frac{1}{2^k}  d_{k+1}  = \frac{2}{n_{i(k+1)}} + O(\frac1{2^{k}})  \to 0$
for $k\to \infty$.

For all $k\geq 1$, let $R_k:= [n_{i(k)}, 2n_{i(k)})  \times  D_k$.
The cardinality of $R_k$ is $n_{i(k)}  \cdot d_k  = 2^ k + O(n_{i(k)})$,
so there exists a bijection $g_k:S_k\to [2^k, 2^{k+1})$ between a subset $S_k$ of $R_k$ and $[2^k, 2^{k+1})$.

Even though the sets $[n_{i(k)}, 2 n_{i(k)})$ are not necessarily disjoint as the $n_{i(k)}$ might not be strictly increasing, the sets $D_k$ and therefore also 
the sets $R_k$ are disjoint.   Hence we may define a function $h\in\Ot$ by setting
$h(x,y) := g_k(x,y)$ whenever $(x,y)\in S_k$ for some $k\geq 1$,
and $h(x,y):=0$ otherwise.  Then $h$ maps
$C\times D$ onto $ \{0\}\cup\bigcup_{k\ge 1 } [ 2^k, 2^{k+1}) = \NN \setminus \{1\}$.
 It remains to check that $h\in\CC$. 

So let $T\subseteq \NN$ be of upper density~$0$, and let $\varepsilon>0$ be a rational number. Note that  
the set~$h[T\times T]$ is the union of the sets~$h[(T\times T)\cap R_k] \subseteq [2^k,2^{k+1})\cup\{0\}$. 
For large enough~$k$ we have
\begin{itemize}
\item on the first coordinate: $|T\cap  [n_{i(k)}, 2n_{i(k)})| \le  n_{i(k)}\varepsilon$;
\item on the second coordinate: $|T\cap   D_k |   \le d_k  =
  \frac{2^k}{n_{i(k)}} + O(1)$.  
\end{itemize}
Hence the cardinality of  $(T\times T) \cap R_k$ 
is bounded by  $$  
n_{i(k)} \cdot \varepsilon \cdot \biggl( \frac{2^k}{n_{i(k)}} + O(1) \biggr) 
 = 2^k \cdot \varepsilon 
\cdot (  1  + O (\frac{n_{i(k)}}{2^k})) \le 2^k \cdot \varepsilon \cdot 2.$$ 
So  $|h[T\times T]\cap [2 ^k, 2 ^{k+1})|\le 2\varepsilon \cdot  2^k$.  As $\varepsilon$ was arbitrary, we can now apply 
 Lemma~\ref{easy.2} to infer that 
$h[T\times T]$ has upper density~$0$.
\end{proof}

The next lemma is again known from \cite{BGHP}, but we give the short proof for the convenience of the reader.
\begin{lem} \label{known.2}
Let $I$ be any ideal on $\NN$, and let $\C_I$ the clone of functions preserving $I$.  If $t\in\Ok$ is a function for which $t [Z^k] = \NN$ for some 
$Z\in I$, then $\langle \C_I\cup \{ t\}\rangle = \OO$.
\end{lem} 
\begin{proof}  Let  $ r = (r_1,\ldots, r_k): \mathbb N\to Z^k$  be a
right inverse of $t$, i.e., $t\circ  r $ is the identity map on $\mathbb N$.  Any function $h:\NN^m\to \NN$ can be written as $h =  t  \circ  ( r  \circ  h )$, 
where each of the functions  $ r_i\circ h: \NN^m\to \NN$ is in $\C_I$, as its range is
a subset of~$Z$.  Hence $h\in \langle \C_I \cup \{t\}\rangle$. 
\end{proof}

\begin{proof}[Proof of Theorem~\ref{thm:precomplete}]
Let $g\notin \CC$.  By Lemma~\ref{known.1} we may assume that 
$g$ is unary.  So there is  a  set $A\subseteq \NN$ of upper density~$0$ such that 
$B:= g[A]$ has positive upper density.   Using Lemma~\ref{n2n}
we  find  a set $C$ of the form $\bigcup_{i\geq 1} [n_i, 2n_i)$, for a strictly increasing sequence $(n_i)_{i\geq 1}$ of natural numbers, and 
a function  $f\in \CC$ such that $f[B] \supseteq C$.  From
Lemma~\ref{ab.lemma}
we get a set $D\subseteq\NN$ of upper density~$0$ and a
 function $h\in \CC$ mapping $C\times D$ onto $\NN$:
$$\begin{array}{ccccccc}
    A & \pfeil{g} &   B&  \pfeil{f}&  C
\\
     &          &     &            & C\times D & \pfeil{h} & \mathbb N 
\end{array}
$$

The assignment $(a,d)\mapsto 
(f(g(a)),d)$ maps $A\times D$ onto  $C\times D$, so 
the binary function $t\in\Ot$ defined
by 
$$  t(x,y) =   h( f(g(x)), y) $$
maps $A\times D$ onto $\NN$.  Clearly $t\in \langle \CC \cup \{ g \}\rangle$.  
Quoting Lemma~\ref{known.2} for $Z:= A\cup D$ now finishes the proof.
\end{proof}


\end{document}